\documentclass[12pt, reqno]{amsart}
\allowdisplaybreaks[1]
\usepackage{amsmath}
\usepackage{amssymb}
\usepackage{amsfonts}
\usepackage{verbatim}
\usepackage[usenames]{color}
\usepackage{hyperref}
\makeindex

 \newtheorem{theorem}{Theorem}[section]
 
 \newtheorem{lemma}[theorem]{Lemma}
 \newtheorem{proposition}[theorem]{Proposition}

\theoremstyle{definition}
\newtheorem{definition}[theorem]{Definition}
\theoremstyle{remark}

\newtheorem{fact*}{Fact}
\newcommand\dd{\mathrm d}

\newcommand{\R}{\mathbb{R}}

\newcommand{\cc}[1]{\overline{#1}}

\newcommand{\ad}{^\ast}
\newcommand{\inv}{^{-1}}

\newcommand{\til}{\raise.17ex\hbox{$\scriptstyle\mathtt{\sim}$}}

\newcommand\beq{\begin{equation}}

\newcommand\eeq{\end{equation}}

\newcommand\black{\color{black}}

\newcommand{\bbm}{\left[ \begin{smallmatrix}}
\newcommand{\ebm}{\end{smallmatrix} \right]}
\newcommand{\bpm}{\left( \begin{smallmatrix}}
\newcommand{\epm}{\end{smallmatrix} \right)}
\numberwithin{equation}{section}

\newcommand{\tensor}[2]{\text{ }{\begin{smallmatrix} #1 \\ \otimes\\ #2\end{smallmatrix}}\text{  }}

\newlength{\Mheight}
\newlength{\cwidth}

\newcommand{\dfn}[1]{{\bf #1}\index{#1}}
\newcommand{\tir}[1]{\tensor{#1}{I}}
\newcommand{\tidr}[1]{\tensor{#1}{\mathrm{id}}}

\newcommand{\CUHP}[1]{\cc\Pi(#1)}
\newcommand{\UHP}[1]{\Pi(#1)}
\newcommand{\BallB}[1]{\mathrm{Ball}(#1)}
\newcommand{\Free}[2]{\mathrm{Free}(#1, #2)}
\newcommand{\MU}[1]{\mathcal{M}(#1)}

\newcommand{\RHPB}[1]{\mathrm{RHP}(#1)}

\title[Free Cauchy transforms]{Cauchy transforms arising from
homomorphic conditional expectations parametrize Free Pick functions but those arising from conditional expectations do not}
\author{
J. E. Pascoe$^\dagger$
}
\address{Department of Mathematics\\
  Washington University in St. Louis\\
  One Brookings Drive \\
 St. Louis, MO 63130}
\email[J. E. Pascoe]{pascoej@math.wustl.edu}
\thanks{$\dagger$ Partially supported by National Science Foundation Mathematical
Science Postdoctoral Research Fellowship  
DMS 1606260}

\author{
Ryan Tully-Doyle
}
\address{Mathematics Department \\
Hampton University\\
Hampton, Virginia 23668}
\email[R. Tully-Doyle]{ryan.tullydoyle@hamptonu.edu}
\date{\today}


\begin{document}

\begin{abstract}
Nevanlinna showed that Cauchy transforms of probability measures parametrize all functions from the upper half plane into itself satisfying a certain asymptotic condition at infinity. 
%Recent work of John D. Williams has shown that an analogous fact is true in the compactly supported case in operator valued free probability.
We show that the correspondence fails in general for the unbounded case for somewhat trivial reasons; however, we show that in a setting of ``homomorphic'' operator valued free probability that Cauchy transforms of
homomorphic conditional expectations parametrize free Pick functions. 
\end{abstract}
\maketitle

\tableofcontents

\section{Introduction}

Classically, R. Nevanlinna proved the following result.

\begin{theorem}[Nevanlinna \cite{nev22}]
Let $\Pi$ denote the upper half plane.
Let $f: \Pi \rightarrow \mathbb{C}.$
The function $f$ is analytic, maps $\Pi$ to $\cc{\Pi}$
and satisfies
	$$\liminf_{s\rightarrow \infty} sf(sz) = -z^{-1},$$
for all $z\in \Pi,$
if and only if 
there exists a probability measure $\mu$ on $\mathbb{R}$
such that
	$$f(z) = \int_\R \frac{1}{t - z} \dd\mu(t).$$
\end{theorem}

Thus, functions with positive imaginary part satisfying good asymptotics are parametrized
by probability measures on the real line.

%We answer a similar inverse problem in the free probabilistic setting in the negative-- however, in an expanded ``homomorphic'' notion of conditional expectation, we show that any function on the noncommutative upper half plane satisfying good asymptotic conditions can be parametrized in this notion of ``homomorphic''  free probability. We make no claims as to the merit of a ``homomorphic'' operator valued  free probability. However, we view that our results suggest that either free function theory is an incomplete method for understanding free probability or that theorems in free probability should extend somewhat trivially to 
%``homomorphic'' operator valued  free probability. That is, there is a bijection between finite measures and 
%holomorphic maps from the upper half plane into itself.

The quantity
	$$f(z) = \int_\R \frac{1}{t - z} \dd\mu(t),$$
occurring in Nevanlinna's theorem is often referred to as
the \dfn{Cauchy transform}.
 Recent work by Anshelevich and Williams \cite{anw14, will13, 2015williams}
has explored the connection between distribution and function theory in free probability in terms of the noncommutative Cauchy transform and the related $R$-transform.
The Cauchy transform and the $R$-transform have served as a vibrant part of free probability, which is evidenced by the large amount of recent work on the subject.

We resolve the correspondence between Cauchy transforms and the class of functions on the upper half plane in the noncommutative context of operator-valued free probability and
free analysis.

\subsection{The noncommutative context}

Let $B$ be a $C^*$-algebra.
The \dfn{matrix universe over $B$}, denoted $\MU{B},$ is the set of square matrices over $B,$ that is
	$$\MU{B} = \bigcup^{\infty}_{n=1} M_n(B).$$
Next, the \dfn{upper half plane over $B$,} denoted $\UHP{B},$ is given by
	$$\UHP{B} =  \{X \in \MU{B} | \hspace{2pt} \text{Im } X > 0\}.$$
Here, we say a self-adjoint operator $A > 0$ if its spectrum is contained in the positive reals and $A \geq 0$ if $A$ has spectrum contained in the non-negative reals.
Similarly, the \dfn{closed upper half plane over $B$,} denoted $\CUHP{B},$ is
	$$\CUHP{B} =  \{X \in \MU{B} | \hspace{2pt} \text{Im } X \geq 0\}.$$
For any $\mathcal{D} \subset \MU{B_1},$ a \dfn{free function} 
$f: \mathcal{D} \rightarrow \MU{B_2}$ is graded and respects intertwining maps. That is, $f$ takes an $n \times n$ matrix over
$B_1$ to an $n \times n$ matrix over $B_2$, and 
if
$\Gamma X = Y \Gamma$ for some rectangular matrix $\Gamma$ of scalars, then  
$\Gamma f(X) = f(Y) \Gamma.$
We denote the set of free functions for $\mathcal{D}$ to $\mathcal{R}$
by  $\Free{\mathcal{D}}{\mathcal{R}}.$ (For more elaborate exposition regarding free analysis, see e.g. the comprehensive presentation in \cite{vvw12}.)

In this noncommutative context, a \dfn{free Pick function} is just a free function $f: \UHP{B_1} \rightarrow \CUHP{B_2}.$ 

Given:
\begin{enumerate}
	\item A $C^*$-algebra $B$,
	\item A von Neumann algebra $M$ unitally containing $B,$
	\item An unbounded self-adjoint operator
	 $A$ affiliated to $M,$ that is, an operator so that each of its spectral projections are contained in
$M,$
	\item A noncommutative conditional expectation 
	$E: M \rightarrow B,$  that is, $E$ is a completely positive unital map satisfying $E(b_1mb_2) = b_1E(m)b_2$
for all $b_1, b_2 \in B$ and
$m \in M$,
\end{enumerate}
we define the \dfn{noncommutative Cauchy transform} of $A$
to be the free function
 $f: \UHP{B} \rightarrow \UHP{B}$ given by the equation
$$f(Z) = \tidr{E}\left(\left(\tir{A}- Z\right)^{-1}\right),$$
where $\mathrm{id}$ denotes the identity map on matrices.
We have adopted a \emph{vertical tensor notation} to save space: $\tensor{A}{B}$
represents the same object as $A \otimes B$.

The obvious analogue of Nevanlinna's theorem would be that
any free function  $f: \UHP{B} \rightarrow \UHP{B}$ satisfying
			$$\lim_{
				\begin{matrix}
					s\rightarrow +\infty \\
					s \in \mathbb{R}
				\end{matrix}						
			} sf(sZ) = -Z^{-1}$$
for all $Z \in \UHP{B}$
would be given by a noncommutative Cauchy transform arising from some $M, E$ and $A$ which could be constructed from $f.$

\emph{The obvious analogue of Nevanlinna's theorem is shown to be
false in Subsection \ref{false}, and thus the ability to reconstruct an algebra, a conditional expectation and an unbounded operator from
a free function $f: \UHP{B} \rightarrow \UHP{B}$ is resolved in the negative.}

However, in an expanded ``homomorphic'' notion of conditional expectation, we show that self maps of the noncommutative upper half plane satisfying good asymptotic conditions are parametrized by Cauchy transforms.

\subsection{Main result}

%\begin{theorem}
%Suppose $\UHP{B_1}$ is an Agler domain.
%Let $f: \UHP{B_1} \rightarrow \CUHP{B_2}$ be a free function.
%The following are equivalent
	%\begin{enumerate}
		%\item
			%$$\liminf_{
				%\begin{matrix}
					%s\rightarrow +\infty \\
					%s \in \mathbb{R}
				%\end{matrix}						
			%} |isf(is)| < \infty.$$
		%\item There exists:
		%\begin{enumerate}
			%\item A von Neumann algebra $M,$
			%\item An unbounded self-adjoint operator $A$ affiliated to $M,$
			%\item A completely positive unital map 
			%$\psi: B_1 \rightarrow M,$
			%\item A completely positive map 
			%$R: M \rightarrow B_2,$
		%\end{enumerate}
		 %so that the function $f$ can be written as
			%$$f(Z) = \tensor{R}{1_n}\left[
			%\left(\tensor{A}{I_n}-
			%\tensor{\psi}{1_n}(Z)\right)^{-1}\right].
			%$$
	%\end{enumerate}
%\end{theorem}
%{\red TYPE I representation.}
%
%{\red Define a a Homomorphic conditional expecation pair}
\begin{definition}
	Let $B$, $M$ be $C^*$-algebras.
	Let $\hat{B}$ be a unital subalgebra of $M.$
	We define a \dfn{homomorphic conditional expectation}
	to be a completely positive unital map
	$E: M \rightarrow B$
	such that $E|_{\hat{B}}$ is a homomorphism.
\end{definition}

The name homomorphic conditional expectation is justified by the following analogue of Tomiyama's theorem \cite{tom57}.
\begin{proposition}[Homomorphic Tomiyama's theorem]
	If $E: \hat{B} \rightarrow M$ is a homomorphic conditional expectation over $B,$ then
		for all $b_1, b_2 \in \hat{B},$  $$E(b_1mb_2) = E(b_1)E(m)E(b_2).$$
\end{proposition}
We prove the above proposition in  Section \ref{proofhomtom}

\begin{definition}
	Let $B$, $\hat{B}$ be $C^*$-algebras.
	We define a \dfn{symmetric dilation} to be a 
	completely positive map 
	$\psi: B \rightarrow \hat{B}$
	so that there exists a $*$-homomorphism
	$E: \hat{B} \rightarrow B$
	such that $E \circ \psi$ is the identity.
\end{definition}

Our main result is as follows.

\begin{theorem}\label{mainresulthomcond}
%Suppose $\UHP{B}$ is an Agler domain.
Let $f: \UHP{B} \rightarrow \CUHP{B}$ be a free function.
The following are equivalent
	\begin{enumerate}
		\item For all $Z \in \UHP{B},$
			$$\lim_{
				\begin{matrix}
					s\rightarrow +\infty \\
					s \in \mathbb{R}
				\end{matrix}						
			} sf(sZ) = -Z^{-1}.$$
		\item There exists:
		\begin{enumerate}
			\item A von Neumann algebra $M,$
			\item A unital subalgebra of $\hat{B} \subseteq M,$
			\item An unbounded self-adjoint operator $A$ affiliated to $M,$
			\item A homomorphic conditional expectation
			$E: M \rightarrow B,$
			\item A symmetric dilation
			$\psi: B \rightarrow \hat{B}$
			such that $E \circ \psi$ is the identity,
		\end{enumerate}
		 so that the function $f$ can be written as
			$$f(Z) = \tensor{E}{\mathrm{id}_n}\left[
			\left(\tensor{A}{I_n}-
			\tensor{\psi}{\mathrm{id}_n}(Z)\right)^{-1}\right].
			$$
	\end{enumerate}
\end{theorem} 
We note that Williams showed that the above theorem holds when $E$ is a conditional expectation and $\psi$ is an identity map if we assume additionally that $f$ has some large analytic continuation at infinity corresponding to the classical compactly supported case \cite{will13}. Our result also generalizes previous results in \cite[Section 5]{pastd14}. In the language of this paper, the representations established in the earlier setting held for $B = \mathbb{C}^m.$

We emphatically take the viewpoint that homomorphic conditional expectations are what makes the Nevanlinna theorem work for noncommutative Cauchy transforms-- we leave to the reader whether or not they generate any deeply interesting analogue of operator-valued free probability.
However, we view that our results suggest that either
${\bf(1)}$ free function theory is an incomplete method for understanding free probability or ${\bf(2)}$ that theorems in free probability should extend somewhat trivially to 
``homomorphic'' operator valued  free probability.

\subsection{Failure of the main result in the usual free probabilistic case}\label{false}
We note that we cannot always reduce to the case where
the symmetric dilation
$\psi$ is the identity map and  $E$ is a \emph{bona fide} conditional expectation.

Take $B = \mathbb{C}^2.$ 
Define $\psi(z_1,z_2) = (z_1,z_2,\frac{1}{2}(z_1+z_2)).$
Define $E(w_1,w_2,w_3) = (w_1,w_2).$
So, we have that $\hat{B} = \mathbb{C}^3.$

Now define $A$ acting on $\mathbb{C}^3$ to satisfy
$A (w_1,w_2,w_3) = A(w_1, w_3, w_2).$

Consider 
$$f(Z) = \tensor{E}{\mathrm{id}_n}\left[
			\left(\tensor{A}{I_n}-
			\tensor{\psi}{\mathrm{id}_n}(Z)\right)^{-1}\right].
			$$

One can show that for $(z_1, z_2)\in \mathbb{C}^2$,
$$f(z_1, z_2) = (-z_1^{-1}, -z_2^{-1}(1-2(z_1+z_2)^{-1}z_2^{-1})).$$ Now we observe that 
$$f(z_1, z_2) = (-z_1^{-1}, 0) + \sum_k (0, -z_2^{-1}[2(z_1+z_2)z_2]^{-k}).$$

If we could choose $\tilde{E}$ a conditional expectation and $\tilde{\psi}$ to be the identity, the homogeneous terms in the above expansion would be polynomials in $z_1^{-1}$ and $z_2^{-1}$ but, evidently, they are not.

\section{Proof of the main result}
We now prove our main theorem, Theorem \ref{mainresulthomcond}.

The \dfn{ball over $B$,} denoted $\BallB{B},$ is the set of contractive matrices over $B,$ that is,
	$$\BallB{B} =  \{X \in \MU{B} | \hspace{2pt} \|X\|<1\}.$$
Similarly, the \dfn{right half plane over $B$,} denoted $\RHPB{B},$ is
	$$\RHPB{B} =  \{X \in \MU{B} | \hspace{2pt} \text{Re } X \geq 0\}.$$

In \cite{ppt16}, the following was proved.
\begin{theorem}[\cite{ppt16}] \label{finalcor}
Let $h: \BallB{B_1} \rightarrow \RHPB{B_2}.$
Then there exists:
\begin{enumerate}
	\item A $C^*$-algebra $M$ unitally containing $B_1,$
	\item	A completely positive linear (not necessarily unital) map $R: M \rightarrow B_2,$ 
	\item A unitary $U \in M,$
	\item A bounded self-adjoint operator $T,$
\end{enumerate}
such that 
\beq \label{herglotzfinal} 
h(X) =\tensor{ iT}{I_n} + \tensor{R}{\mathrm{id}_n}
\left[\left( I + \tensor{U}{I_n}X\right)\left(I - \tensor{U}{I_n}X\right)^{-1}\right].
\eeq
\end{theorem}
We note that although the statement in \cite[Corollary 3.6]{ppt16} assumes an exactness hypothesis on $B_1$, recent advances in Agler model theory by Ball, Marx and Vinnikov in the preprint \cite[Corollary 3.2]{BMV16} give the full result by \cite[Lemma 3.3]{ppt16}.

We use Theorem \ref{finalcor} to show the following Nevanlinna representation via a Hilbert space geometric derivation.
\begin{theorem}
%Suppose $\UHP{B_1}$ is an Agler domain.
Let $f: \UHP{B_1} \rightarrow \CUHP{B_2}$ be a free function.
The following are equivalent
	\begin{enumerate}
		\item
			$$\liminf_{
				\begin{matrix}
					s\rightarrow +\infty \\
					s \in \mathbb{R}
				\end{matrix}						
			} |isf(is)| < \infty.$$
		\item There exists:
		\begin{enumerate}
			\item A von Neumann algebra $M,$
			\item An unbounded self-adjoint operator $A$ affiliated to $M,$
			\item A completely positive unital map 
			$\psi: B_1 \rightarrow M,$
			\item A completely positive map 
			$R: M \rightarrow B_2,$
		\end{enumerate}
		 so that the function $f$ can be written as
			$$f(Z) = \tensor{R}{\mathrm{id}_n}\left[
			\left(\tensor{A}{I_n}-
			\tensor{\psi}{\mathrm{id}_n}(Z)\right)^{-1}\right].
			$$
	\end{enumerate}
\end{theorem}

\begin{proof}
We adopt the technique used in the proof of a general Nevanlinna types theorem as in \cite{aty13,pastd14}.

Let $f$ be as in the statement of the Theorem. By concretely realizing Theorem \ref{finalcor}, we can instantiate a Herglotz function $h$ which satisfies
$ih((Z + i)\inv (Z - i))=f(Z) - T$ for some self-adjoint $T$. By Theorem \ref{finalcor}, $h$ can be written concretely as
  $$h(\Lambda) = \tir{V\ad} \left(\tensor{L}{I} - \Lambda\right)\inv \left(\tensor{L}{I} + \Lambda\right) \tir{V}.$$
(Here we have concretely written $R(x) = V\ad x V$ and used a resolvent of the form $(L - X)^{-1}(L+X)$ instead of $(1-UX)^{-1}(1+UX)$ to agree with \cite{aty13, ag90}. However, $L$ is still a unitary. If fact, the algebra will show that $L = U\ad.$)

Let 
	$$f(Z)-\tir{T} = i\tir{V\ad}\left(\tensor{L}{I} - {(Z + i)\inv (Z - i)}\right)\inv \left(\tensor{L}{I} + {(Z + i)\inv (Z - i)}\right)\tir{V}.$$
	
One can show as an elementary exercise in the spectral theorem that every vector of the form $Vw$ is in the domain of the normal inverse
$(1-L)^{-1}.$ Notably, this reduces to an exercise in measure theory and manipulation of classical Herglotz integrals.
\begin{lemma}
Any vector of the form $Vw$ is in the domain of the normal inverse of $(1-L)^{-1}.$ Namely, the range of $V$
is perpendicular to the kernel of $1-L.$
\end{lemma}
\begin{proof}
Consider our function
		$$f(Z)-\tir{T} = i\tir{V\ad}\left(\tensor{L}{I} - {(Z + i)\inv (Z - i)}\right)\inv \left(\tensor{L}{I} + {(Z + i)\inv (Z - i)}\right)\tir{V}.$$
	Evaluate at $Z = is.$
	$$f(is)-T = iV\ad\left(L - {(is + i)\inv (is - i)}\right)\inv \left( L+ {(is + i)\inv (is - i)}\right)V.$$
	
	So, since $L$ is unitary and thus normal, evaluating $w^*(f(is)-T)w$ gives, via the the spectral theorem,
	\begin{align*}
	w^*(f(is)-T)w& = iw^*V\ad\left(L - {(is + i)\inv (is - i)}\right)\inv \left( L+ {(is + i)\inv (is - i)}\right)Vw
	\\&= i\int_{\mathbb{T}} 
		\frac{\omega + {(is + i)\inv (is - i)}}
		{\omega - {(is + i)\inv (is - i)}}
	 d\mu_{Vw}(\omega)
	\\& = i\int_{\mathbb{T}} 
		\frac{\omega(s+1) + (s - 1)}
		{\omega(s+1) - (s-1)}
	d \mu_{Vw}(\omega).
\end{align*}		
	
Note that the condition 
$$\liminf_{
				\begin{matrix}
					s\rightarrow +\infty \\
					s \in \mathbb{R}
				\end{matrix}						
			} |isf(is)| < \infty.$$
implies \emph{a fortiori} that
$$\liminf_{
				\begin{matrix}
					s\rightarrow +\infty \\
					s \in \mathbb{R}
				\end{matrix}						
			}  s\text{Im }f(is) < \infty.$$
			
So, consider		
	\begin{align*}s\text{Im }w^*f(is)w & =  s\text{Im }w^*(f(is)-T)w
 \\&=  s\text{Im }i\int_{\mathbb{T}} 
		\frac{\omega(s+1) + (s - 1)}
		{\omega(s+1) - (s-1)}
	 d\mu_{Vw}(\omega)
	   \\&= \int_{\mathbb{T}} 
		\frac{s^2}
		{s^2 + 1 - (s^2 - 1)\text{Re }\omega}
	 d\mu_{Vw}(\omega)
	.\end{align*}
As $s$ goes to infinity,
noting that the integrand is monotone increasing in $s$, by monotone convergence theorem 
		$$\int_{\mathbb{T}} 
		\frac{1}
		{1 - \text{Re }\omega}
	 d\mu_{Vw}(\omega) 
	  = \liminf_{s\to\infty} s\text{Im }w^*f(is)w < \infty.$$
	  
Since
$$\int_{\mathbb{T}} 
		\frac{1}
		{1 - \text{Re }\omega}
	 d\mu_{Vw}(\omega)  = \int_{\mathbb{T}} 
		\frac{2}
		{|1 - \omega|^2}
	d \mu_{Vw}(\omega),$$
we are done, because $Vw$ is the domain of $f(L)$
if and only if $|f|^2$ is integrable with respect to $d\mu_{Vw}.$
\end{proof}

Straightforward algebra gives
	\begin{align*}
		f(Z)-\tir{T} &= i\tir{V\ad}\left(\tensor{L}{I} - {(Z + i)\inv (Z - i)}\right)\inv \left(\tensor{L}{I} + {(Z + i)\inv (Z - i)}\right)\tir{V} \\
		&= i\tir{V\ad}\left({(Z+i)} \tensor{L}{I} - {(Z - i)} \right)\inv \left({(Z + i)} \tensor{L}{I} + {(Z - i)}\right)\tir{V} \\
		&= i\tir{V\ad}\left({Z} \tensor{L - I}{I} + i\tensor{L+I}{I}\right)\inv \left( {Z} \tensor{L+I}{I} - i \tensor{L - I}{I} \right)\tir{V}.
	\end{align*}

Decompose $L$ into blocks acting on $\ker 1 - L$ and $\ker(1-L)^\perp$ as $$L = \bbm 1 & 0 \\ 0 & L_0 \ebm$$ so that $\ker 1 - L_0$ is trivial. Multiply through on the left by $I = \bbm 1 & 0 \\ 0 & (1- L_0)\inv \ebm \bbm 1 & 0 \\ 0 & 1- L_0 \ebm$.
We get
\begin{align*}
 &i\tir{V\ad}
 \tir{\bbm 1 & 0 \\ 0 & (1- L_0)\inv \ebm}
 \tir{\bbm 1 & 0 \\ 0 & 1 - L_0 \ebm}
 \left(Z\tir{ (L- I)} +\tir{ i(L +I)}\right)\inv \left(Z\tir{(L+I)} -\tir{ i (L - i)}\right)
 \tir{V} \\
 &=i\tir{V\ad} \tir{\bbm 1 & 0 \\ 0 & (1- L_0)\inv \ebm}
  \left(Z \tir{(L - I)}\tir{ \bbm 1 & 0 \\ 0 & (1 - L_0)\inv \ebm} + \tir{i (L + I)}\tir{\bbm 1 & 0 \\ 0 & (1 - L_0)\inv \ebm}\right)\inv \\ &\hspace{1in} \times \left(Z\tir{(L + I)} -\tir{ i (L - i)}\right)\tir{V}\\
 &=i\tir{V\ad} 
 \tir{\bbm 1 & 0 \\ 0 & (1- L_0)\inv\ebm}
  \left(Z \tir{\bbm 0 & 0 \\ 0 & -1 \ebm} + \tir{i \bbm 2 & 0 \\ 0 & (1 + L_0) (1 - L_0)\inv \ebm}\right)\inv  \left(Z\tir{(L+I)} -\tir{ i (L - I)}\right) \tir{V}\\
&= i\tir{V\ad}\tir{\bbm 1 & 0 \\ 0 & (1- L_0)\inv \ebm}
 \left(Z \tir{\bbm 0 & 0 \\ 0 & -1 \ebm }+ \tir{i \bbm 2 & 0 \\ 0 & (1 + L_0) (1 - L_0)\inv \ebm}\right)\inv \\ &\hspace{1in} \times \left(Z\tir{\bbm 2 & 0 \\ 0 & L_0 + 1\ebm} - \tir{ i \bbm 0 & 0 \\ 0 & L_0 - I \ebm}\right)\tir{V}
\end{align*}

The operator $$A = i \frac{1 + L_0}{1 - L_0}$$ is a densely defined self-adjoint unbounded operator since $L_0$ has no kernel.% (see argument in \cite{aty13}).

Since we are only interested in $V\ad M(Z) V$, the upper triangular form of
$$\left(Z \tir{\bbm 0 & 0 \\ 0 & -1 \ebm} + \tir{i \bbm 2 & 0 \\ 0 & (1 + L_0) (1 - L_0)\inv \ebm}\right)\inv $$ and the structure of $V$, namely that $V$ is perpendicular to the kernel of $1-L$ , gives that the relevant operator is the (2,2) block. Then compress $Z$ to  $$\tidr{\psi}(Z) = Z_\psi=\tir{P}Z\tir{P^*}$$
where $P$ is the projection onto the perp of the kernel of 
$I-L$.
Then our resolvent has the form

%$$iM(Z) = (1 - L)\inv (A - Z)\inv(iZ(L + I)  + (L - I)).$$

\begin{align*}
&f(Z) - \tir{T} = \tir{V\ad (I - L_0)\inv}
 \left(\tir{A} - Z_\psi\right)\inv
 \left(iZ_\psi\tir{(L_0 + I)}  + \tir{(L_0 - I)}\right)\tir{V} \\
&=\tir{V\ad (I - L_0)\inv} 
\left(\tir{A} - Z_\psi\right)\inv
\left(iZ_\psi\tir{(L_0 + I)(I - L_0)\inv}  + I \right)\tir{(I - L_0) V}\\
&= \tir{V\ad (I- L_0)\inv} \left(\tir{A} - Z_\psi\right)
\inv \left(Z_\psi \tir{A} + I\right) \tir{(I - L_0) V}\\
&=\tir{V\ad  (I - L_0)\inv} \left(\tir{A} - Z_\psi\right)\inv \left(Z_\psi \tir{A} - \tir{A^2} + \tir{A^2} + I\right)\tir{ (I - L_0) V} \\
&= \tir{V\ad (I - L_0)\inv} \left(\tir{A} - Z_\psi\right)\inv 
\left[\left(Z_\psi - \tir{A})\tir{A} + (\tir{A^2} + I\right)\right] \tir{(I - L_0) V}\\
&= \tir{V\ad (I - L_0)\inv A (I - L_0)V} + \tir{V(I - L_0)\inv} \left(\tir{A} - Z_\psi\right)\inv \left(\tir{A^2} + I\right) \tir{(I - L_0)V} \\
&= \tir{V\ad AV} + \tir{V\ad (I - L_0)\inv }\left(\tir{A} - Z_\psi\right)\inv \left(\tir{A^2} + I\right) \tir{(I - L_0)V}\\
&= \tir{V\ad AV} +\tir{V\ad (I - L_0)\inv }\left(\tir{A} - Z_\psi\right)\inv  \tir{(I - L_0\ad)^{-1}V}.
\end{align*}

Now, the asymptotic condition implies that the constant terms must vanish, so
$$f(Z) = \tir{V\ad (I - L)\inv }\left(\tir{A} - Z_\psi\right)\inv  \tir{(I - L\ad)^{-1}V}.$$

Defining a new $R(x) = V\ad (I - L)\inv x (I - L\ad)\inv V  $ and $\psi$ to be as above,
we are done.

%The operator $$A = i \frac{1 + L_0}{1 - L_0}$$ is a densely defined self-adjoint unbounded operator (see \cite{}).

\black

\end{proof}

The main result Theorem \ref{mainresulthomcond} now follows by noting that
	$$E(-\psi(Z)^{-1}) = \lim_{
				\begin{matrix}
					s\rightarrow +\infty \\
					s \in \mathbb{R}
				\end{matrix}						
			} sf(sZ) = -Z^{-1}.$$

So we see that 
	$$E(\psi(Z)^{-1}) = Z^{-1}.$$

One can show that
	$$E(\psi(H_1)\ldots\psi(H_k)) = H_1 \ldots H_k,$$
by taking $Z = I_{k+1} - H,$ where $H$ has $H_1, \ldots, H_k$
on the upper diagonal.
\begin{lemma}
		$$E(\psi(H_1)\ldots\psi(H_k)) = H_1 \ldots H_k.$$
\end{lemma}
\begin{proof}
	
	Note
	$$(I - H)^{-1} = \sum^{\infty}_{i=0} H^i,$$
	and
	$$\tidr{E}\left(\tidr{\psi}(I - H)^{-1}\right)
	= \sum^{\infty}_{i=0} \tidr{E}\left(\left[\tidr{\psi}(H)\right]^i\right).$$
	
	So we obtain that
		$\tidr{E}\left(\left[\tidr{\psi}(H)\right]^k\right) = H^k.$
	Evaluating at
			$$H = \bpm 0 & H_1 \\
			& \ddots & \ddots &\\
			& & 0 & H_k \\
			& &  & 0 \epm$$
	and looking at the block $(1, k+1)$ entry
	gives the claim.
\end{proof}
Now, we obtain the necessary homomorphic properties by
letting $\hat{B}$ be the algebra generated
the range of $\psi, $ so  we are done.

 \section{Proof of the homomorphic Tomiyama's theorem}
 \label{proofhomtom}
 We now prove our analog of Tomiyama's theorem for homomorphic conditional expectations.
\begin{proof}
Our proof follows Tomiyama's original method in \cite{tom57}.

Suppose $E$ is a homomorphic conditional expectation.
Without loss of generality, assume all $C^*$-algebras involved are weakly closed. (That is, we can extend everything with the Stinespring theorem.) It is sufficient to show that for any projection $e$ in $\hat{B}$
we have that $E(em) = E(e)E(m).$

Let $e$ be a projection in $\hat{B}.$ Let $x$ be a
positive element of $M.$
Note that
$$E(exe) \leq E(e\|x\| e) = \|x\|E(e).$$
So, since $E(e)$ is a projection by the homomorphic property, $$E(exe) = E(e)E(exe)E(e).$$

Now with a general element $m \in M,$ 
	$$ 0 \leq \bpm 1 & 0 \\ 0 & 1- E(e)\epm E\bpm 1 & me
	\\ em^* & emm^*e \epm \bpm 1 & 0 \\ 0 & 1- E(e)\epm 
	= \bpm 1 & E(m^*e)(1-E(e)) \\ (1-E(e))E(em) & 0\epm,  $$
and so $$(1-E(e))E(em) = 0.$$

Thus, $E(em) = E(e)E(em) = 
E(e)E(em) + E(e)E((1-e)m)= E(e)E(m)$
and we are done.
\end{proof}

\bibliography{references}

\begin{thebibliography}{10}

\bibitem{ag90}
J.~Agler.
\newblock On the representation of certain holomorphic functions defined on a
  polydisc.
\newblock In {\em Operator Theory: Advances and Applications, {Vol. 48}}, pages
  47--66. {Birkh\"auser}, Basel, 1990.

\bibitem{aty13}
J.~Agler, R.~Tully-Doyle, and N.J. Young.
\newblock Nevanlinna representations in several variables.
\newblock {\em Journal of Functional Analysis}, 270(8):3000 -- 3046, 2016.

\bibitem{anw14}
Michael Anshelevich and John~D. Williams.
\newblock Operator-valued monotone convolution semigroups and an extension of
  the {B}ercovici-{P}ata bijection.
\newblock arXiv: 1412.1413, 2014.

\bibitem{BMV16}
J.~A. Ball, Gregory Marx, and Victor Vinnikov.
\newblock Interpolation and transfer-function realization for the
  noncommutative {S}chur-{A}gler class.
\newblock arXiv:1602.00762, 2016.

\bibitem{vvw12}
D.~S. {Kaliuzhnyi-Verbovetskyi} and V.~{Vinnikov}.
\newblock {Foundations of Noncommutative Function Theory}.
\newblock {\em ArXiv e-prints}, 2012.

\bibitem{nev22}
R.~Nevanlinna.
\newblock {\ Asymptotisch Entwicklungen beschr\"ankter Funktionen und das
  Stieltjessche Momentproblem}.
\newblock {\em Ann. Acad. Sci. Fenn. Ser. A}, 18, 1922.

\bibitem{ppt16}
J.~E. Pascoe, B.~Passer, and R.~Tully-Doyle.
\newblock Representations of {H}erglotz functions.
\newblock preprint: arXiv:1607.00407, 2016.

\bibitem{pastd14}
J.~E. Pascoe and Ryan Tully-Doyle.
\newblock Free {P}ick functions: representations, asymptotic behavior and
  matrix monotonicity in several noncommuting variables.
\newblock submitted for publication. arXiv: 1309.1791, 2014.

\bibitem{tom57}
J.~Tomiyama.
\newblock On the projection of norm one in {W$\ast$}-algebras.
\newblock {\em Proc. Japan Acad.}, 33(10):608--612, 1957.

\bibitem{will13}
John~D. Williams.
\newblock Analytic function theory for operator-valued free probability.
\newblock to appear in Crelle's journal, 2013.

\bibitem{2015williams}
John~D. Williams.
\newblock {B-Valued Free Convolution for Unbounded Operators}.
\newblock To appear in IUMJ, July 2015.

\end{thebibliography}
\bibliographystyle{plain}

\end{document}